\documentclass[11pt,letter]{article}
\input{style}

\begin{document}

\title{Integral representation of martingales motivated by the problem
  of endogenous completeness in financial economics}

\author{Dmitry Kramkov\thanks{This research was supported in part by
    the Carnegie Mellon-Portugal Program and by the Oxford-Man
    Institute for Quantitative Finance at the University of Oxford.}\\
  Carnegie Mellon University and
  University of Oxford,\\
  Department of Mathematical Sciences,\\ 5000 Forbes Avenue,
  Pittsburgh, PA, 15213-3890, USA \and Silviu Predoiu\thanks{This
    research was part of PhD thesis completed
    at Carnegie Mellon University in June 2011.}\\
  Citigroup, New York, USA}

\date{\today}

\maketitle
\begin{abstract}
  Let $\mathbb{Q}$ and $\mathbb{P}$ be equivalent probability measures
  and let $\psi$ be a $J$-dimensional vector of random variables such
  that $\frac{d\mathbb{Q}}{d\mathbb{P}}$ and $\psi$ are defined in
  terms of a weak solution $X$ to a $d$-dimensional stochastic
  differential equation. Motivated by the problem of \emph{endogenous
    completeness} in financial economics we present conditions which
  guarantee that every local martingale under $\mathbb{Q}$ is a
  stochastic integral with respect to the $J$-dimensional martingale
  $S_t \set \mathbb{E}^{\mathbb{Q}}[\psi|\mathcal{F}_t]$. While the
  drift $b=b(t,x)$ and the volatility $\sigma = \sigma(t,x)$
  coefficients for $X$ need to have only minimal regularity properties
  with respect to $x$, they are assumed to be analytic functions with
  respect to $t$. We provide a counter-example showing that this
  $t$-analyticity assumption for $\sigma$ cannot be removed.
\end{abstract}

\begin{description}
\item[AMS 2010 subject classifications:] Primary 60G44, 91B51, 91G99;
  secondary 35K15, 35K90.
\item[Key words:] integral representation, martingales, diffusion,
  parabolic equations, analytic semigroups, real analytic functions,
  Krylov-Ito formula, dynamic completeness, equilibrium.
\end{description}

\section{Introduction}
\label{sec:introduction}

Let $(\Omega, \mathcal{F}_1, \mathbf{F} = (\mathcal{F}_t)_{t\in
  [0,1]}, \mathbb{P})$ be a complete filtered probability space,
$\mathbb{Q}$ be an equivalent probability measure, and $S = (S^j_t)$
be a $J$-dimensional martingale under $\mathbb{Q}$. It is often
important to know whether every local martingale $M = (M_t)$ under
$\mathbb{Q}$ admits an integral representation with respect to $S$,
that is,
\begin{equation}
  \label{eq:1}
  M_t = M_0 + \int_0^t H_u dS_u, \quad t\in [0,1], 
\end{equation}
for some predictable $S$-integrable process $H = (H^j_t)$. For
instance, in mathematical finance, which is the topic of a particular
interest to us, the existence of such a martingale representation
corresponds to the \emph{completeness} of the market model driven by
stock prices $S$, see \citet{HarrPlis:83}.

A general answer is given in~\citet[Section XI.1(a)]{Jacod:79}.
Jacod's theorem states that the integral representation property holds
if and only if $\mathbb{Q}$ is the unique equivalent martingale
measure for $S$. In mathematical finance this result is sometimes
referred to as the 2nd fundamental theorem of asset pricing.

In many applications, the process $S$ is defined in a \emph{forward
  form}, in terms of its predictable characteristics under
$\mathbb{P}$. The density process $Z$ of a martingale measure
$\mathbb{Q}$ for $S$ is then constructed through the use of the
Girsanov theorem and its generalizations, see \citet{JacodShir:02}.
The verification of the existence of integral representations for all
$\mathbb{Q}$-martingales under $S$ is often straightforward. For
example, if $S$ is a diffusion process under $\mathbb{P}$ with the
drift vector-process $b =(b_t)$ and the volatility matrix-process
$\sigma = (\sigma_t)$, then such a representation exists if and only
if $\sigma$ has full rank $d\mathbb{d\mathbb{P}} \times dt$ almost
surely.

In this paper we assume that both $S$ and $Z$ are described in a
\emph{backward form}, through their terminal values. Given random
variables $\xi>0$ and $\psi = (\psi^j)_{j=1,\dots,J}$ they are defined
as
\begin{gather*}
  Z_1 \set \frac{d\mathbb{Q}}{d\mathbb{P}} \set
  \frac{\xi}{\mathbb{E}[\xi]}, \\
  S_t \set \mathbb{E}^{\mathbb{Q}}[\psi|\mathcal{F}_t], \quad t\in
  [0,1].
\end{gather*}
We are looking for (easily verifiable) conditions on $\xi$ and $\psi$
guaranteeing the integral representation of all
$\mathbb{Q}$-martingales with respect to $S$.

Our work is motivated by the problem of \emph{endogenous completeness}
in continuous-time financial economics which naturally arises in the
construction of Radner equilibrium, see \citet*{AnderRaim:08},
\citet*{HugMalTrub:12}, and \citet*{RiedHerz:12}, and in the study of
the equilibrium-based price impact models, see \citet*{BankKram:11b}
and \citet*{Germ:11}. Here $\xi$ is an equilibrium state price
density, usually defined implicitly by a fixed point argument, and
$\psi=(\psi^j)$ is the random vector of the cumulative discounted
dividends for traded stocks.  The term ``endogenous'' is used because
the stock prices $S$ are \emph{computed} as an output of
equilibrium. A similar problem also arises in the verification of the
completeness of markets where, in addition to stocks, one can also
trade options, see \citet*{DavisObjoj:08}.

We focus on the case when $\xi$ and $\psi$ are defined in terms of a
weak solution $X$ to a $d$-dimensional stochastic differential
equation. With respect to $x$ the coefficients of this equation
satisfy classical conditions guaranteeing weak existence and
uniqueness: the drift vector $b(t,\cdot)$ is measurable and bounded
and the volatility matrix $\sigma(t,\cdot)$ is uniformly continuous
and bounded and has a bounded inverse. With respect to $t$ our
assumptions are more stringent: $b(\cdot,x)$ and $\sigma(\cdot,x)$ are
required to be analytic functions on $(0,1)$. We give an example
showing that this $t$-analyticity assumption on $\sigma$ cannot be
removed.

Our results complement those in \cite{AnderRaim:08},
\cite{HugMalTrub:12}, and \cite{RiedHerz:12}. The primary focus of
these papers is the existence of Radner equilibrium. The ``backward''
martingale representation is used implicitly as a link between static
and dynamic equilibrium.

In the pioneering paper \cite{AnderRaim:08}, $X$ is a Brownian
motion. The proofs in this paper rely on non-standard analysis. In
\cite{HugMalTrub:12} the conditions are imposed on the diffusion
coefficients $b=b(t,x)$ and $\sigma = \sigma(t,x)$ and on the
transition density $p=p(t,x,s,y)$. In the main body
of~\cite{HugMalTrub:12}, it is assumed that $b$, $\sigma$, and $p$ are
analytic functions with respect to all their arguments.  In the
technical appendix to \cite{HugMalTrub:12}, these functions are
required to be analytic with respect to $t$ and $s$ and $5$-times
($7$-times for $p$) continuously differentiable with respect to $x$
and $y$. In \cite{RiedHerz:12} the diffusion coefficients $b$ and
$\sigma$ do not depend on $t$, the matrix $\sigma$ is invertible and
$b=b(x)$, $\sigma=\sigma(x)$, and $\sigma^{-1} = \sigma^{-1}(x)$ and
bounded and analytic functions.

In one important aspect, the assumptions in \cite{AnderRaim:08} and
\cite{RiedHerz:12} and in the main body of \cite{HugMalTrub:12} are
less restrictive than those in this paper. If $\psi = F(X_1)$, where
$F = F^j(x)$ is a $J$-dimensional vector-function on $\mathbb{R}^d$,
then these papers require the Jacobian matrix of $F$ to have rank $d$
only on some open set. In our framework and also in the setup of the
technical appendix to \cite{HugMalTrub:12}, this property needs to
hold almost everywhere on $\mathbb{R}^d$. We provide an example
showing that in the absence of the $x$-analyticity assumption on
$b=b(t,x)$ and $\sigma=\sigma(t,x)$ this stronger condition cannot be
relaxed.

From the point of view of applications, the most severe constraint of
our setup is the boundedness assumption on the diffusion coefficients.
This condition facilitates references to the results from elliptic
PDEs commonly stated for bounded coefficients; our main source is
\citet*{Kryl:08}. At the same time, it excludes some popular financial
modes such as those driven by an Ornstein-Uhlenbeck process. We leave
the extension to unbounded coefficients for future research.

In \cite{Kram:12b} the results of this paper are used to obtain
criteria for the existence of dynamic Radner equilibrium.

\section{Main results}
\label{sec:main-results}

Let $\mathbf{X}$ be a Banach space with the norm
$\norm{\cdot}_{\mathbf{X}}$.  We shall often use maps
$\map{f}{[0,1]}{\mathbf{X}}$ which are analytic on $(0,1)$ and
H\"older continuous on $[0,1]$, that is, for every $t\in (0,1)$ there
exist a number $\epsilon(t)>0$ and a sequence $(A_n(t))_{n\geq 0}$ in
$\mathbf{X}$ such that
\begin{displaymath}
  f(s) = \sum_{n=0}^\infty A_n(t) (s-t)^n, \quad s\in (0,1),
  \abs{s-t}<\epsilon(t), 
\end{displaymath}
and there are constants $N>0$ and $\delta>0$ such that
\begin{displaymath}
  \norm{f(t)-f(s)} \leq N \abs{t-s}^{\delta}, \quad
  s,t\in [0,1].  
\end{displaymath}
In the statements of our main results, $\mathbf{X}$ is one of the
following spaces:
\begin{description}
\item[$\mathbf{L}_\infty=\mathbf{L}_\infty(\mathbb{R}^d,dx)$\rm{:}]
  the Lebesgue space of bounded real-valued functions $f$ on
  $\mathbb{R}^d$ with the norm $\norm{f}_{\mathbf{L}_\infty} \set
  \esssup_{x\in \mathbb{R}^d} \abs{f(x)}$.
\item[$\mathbf{C}=\mathbf{C}(\mathbb{R}^d)$\rm{:}] the Banach space of
  bounded and continuous real-valued functions $f$ on $\mathbb{R}^d$
  with the norm $\norm{f}_{\mathbf{C}} \set \sup_{x\in \mathbb{R}^d}
  \abs{f(x)}$.
\end{description}

We use standard notations of linear algebra. If $x$ and $y$ are
vectors in $\mathbb{R}^n$, then $\ip{x}{y}$ denotes the scalar product
and $\abs{x} \set \sqrt{\ip{x}{x}}$. If $a\in \mathbb{R}^{m\times n}$
is a matrix with $m$ rows and $n$ columns, then $ax$ denotes its
product on the (column-)vector $x$, $a^*$ stands for the transpose,
and $\abs{a} \set \sqrt{\trace(aa^*)}$.

Let $\mathbb{R}^d$ be an Euclidean space and the functions
$\map{b=b(t,x)}{[0,1]\times\mathbb{R}^d}{\mathbb{R}^d}$ and
$\map{\sigma=\sigma(t,x)}{[0,1]\times
  \mathbb{R}^d}{\mathbb{R}^{d\times d}}$ be such that for all
$i,j=1,\dots,d$:
\begin{enumerate}[label=(A\arabic{*}), ref=(A\arabic{*})]
\item \label{item:1} the maps $t\mapsto b^i(t,\cdot)$ of $[0,1]$ to
  $\mathbf{L}_{\infty}$ and $t\mapsto \sigma^{ij}(t,\cdot)$ of $[0,1]$
  to $\mathbf{C}$ are analytic on $(0,1)$ and H\"older continuous on
  $[0,1]$.  For $t\in [0,1]$ and $x\in \mathbb{R}^d$ the matrix
  $\sigma(t,x)$ has the inverse $\sigma^{-1}(t,x)$ and there exists a
  constant $N>0$, same for all $t$ and $x$, such that
  \begin{equation}
    \label{eq:2}
    \abs{\sigma^{-1}(t,x)} \leq N.
  \end{equation} 
  Moreover, there exists a strictly increasing function $\omega =
  (\omega(\epsilon))_{\epsilon>0}$ such that $\omega(\epsilon)\to 0$
  as $\epsilon\downarrow 0$ and, for all $t\in [0,1]$ and all $x,y\in
  \mathbb{R}^d$,
  \begin{displaymath}
    \abs{\sigma(t,x) - \sigma(t,y)} \leq \omega(\abs{x-y}).
  \end{displaymath}
  \setcounter{item}{\value{enumi}}
\end{enumerate}
Note that~\eqref{eq:2} is equivalent to the uniform ellipticity of the
covariance matrix-function $a\set \sigma\sigma^*$: for all $y\in
\mathbb{R}^d$ and $(t,x)\in [0,1]\times \mathbb{R}^d$,
\begin{displaymath}
  \ip{y}{a(t,x)y} = \abs{\sigma(t,x) y}^2 \geq \frac{1}{N^2}\abs{y}^2.  
\end{displaymath}

Let $X_0\in \mathbb{R}^d$. The classical results of
\citet[Theorem~7.2.1]{StrVarad:06} and \citet{Kryl:69,Kryl:72} imply
that under~\ref{item:1} there exist a complete filtered probability
space $(\Omega, \mathcal{F}_1, \mathbf{F} = (\mathcal{F}_t)_{t\in
  [0,1]}, \mathbb{P})$, a Brownian motion $W$, and a stochastic
process $X$, both with values in $\mathbb{R}^d$, such that
\begin{equation}
  \label{eq:3}
  X_t = X_0+\int_0^t b(s,X_s)ds + \int_0^t\sigma(s,X_s) dW_s, \quad t\in [0,1],  
\end{equation}
and, moreover, all finite dimensional distributions of $X$ are defined
uniquely. In view of~\eqref{eq:2}, we can (and will) assume that the
filtration $\mathbf{F}$ is generated by $X$:
\begin{equation}
  \label{eq:4}
  \mathbf{F} = \mathbf{F}^X \set (\mathcal{F}^X_t)_{t\in [0,1]},
\end{equation}
where, as usual, $\mathcal{F}_t^X$ denotes the $\sigma$-field
generated by $(X_s)_{s\leq t}$ and complemented with $\mathbb{P}$-null
sets. In this case, $\mathbb{P}$ is defined uniquely in the sense that
if $\mathbb{Q}\sim \mathbb{P}$ is an equivalent probability measure on
$(\Omega,\mathcal{F}_1) = (\Omega,\mathcal{F}^X_1)$ such that
\begin{displaymath}
  W_t = \int_0^t \sigma^{-1}(s,X_s) (dX_s - b(s,X_s)ds), \quad t\in [0,1], 
\end{displaymath}
is a Brownian motion under $\mathbb{Q}$, then $\mathbb{Q} =
\mathbb{P}$. Note that the filtration $\mathbf{F}^X$ is (left- and
right-) continuous because every $\mathbf{F}^X$-martingale is
continuous, see Remark~\ref{rem:2}.

\begin{Remark}
  \label{rem:1}
  With respect to $x$, the conditions in \ref{item:1} are,
  essentially, the minimal classical assumptions guaranteeing the
  existence and the uniqueness of the weak solution
  to~\eqref{eq:3}. This weak solution is also well-defined when $b$
  and $\sigma$ are only measurable functions with respect to $t$.  As
  we shall see in Example~\ref{ex:1}, the requirement on $\sigma$ to
  be $t$-analytic is, however, essential for the validity of our main
  Theorem~\ref{th:1}.
\end{Remark}

\begin{Remark}
  \label{rem:2}
  It is well-known that a local martingale $M$ adapted to the
  filtration $\mathbf{F}^W$, generated by the Brownian motion $W$, is
  a stochastic integral with respect to $W$, that is, there exists an
  $\mathbf{F}^W$-predictable process $H$ with values in $\mathbb{R}^d$
  such that
  \begin{equation}
    \label{eq:5}
    M_t = M_0 + \int_0^t H_u dW_u \set M_0 + \sum_{i=1}^d \int_0^t
    H^i_u dW^i_u, \quad t\in [0,1].  
  \end{equation}
  The example in \citet{Barl:82} shows that under~\ref{item:1} the
  filtration $\mathbf{F}^W$ may be strictly smaller than
  $\mathbf{F}=\mathbf{F}^X$. Nevertheless, for a local martingale $M$
  adapted to $\mathbf{F}$ the integral representation~\eqref{eq:5}
  still holds with some $\mathbf{F}$-predictable $H$. This follows
  from the mentioned fact that a probability measure $\mathbb{Q}\sim
  \mathbb{P}$ such that $W$ is a $\mathbb{Q}$-local martingale
  (equivalently, a $\mathbb{Q}$-Brownian motion) coincides with
  $\mathbb{P}$ and the integral representation theorems in
  \citet[Section XI.1(a)]{Jacod:79}.
\end{Remark}

Recall that a locally integrable function $f$ on $(\mathbb{R}^d,dx)$
is \emph{weakly differentiable} if for every index $i=1,\dots,d$ there
is a locally integrable function $g^i$ such that the identity
\begin{displaymath}
  \int_{\mathbb{R}^d} g^i(x)h(x) dx = -
  \int_{\mathbb{R}^d} f(x) \frac{\partial h}{\partial x^i}(x) dx 
\end{displaymath}
holds for every function $h\in \mathbf{C}^\infty$ with compact
support, where $\mathbf{C}^\infty$ is the space of infinitely many
times differentiable functions. In this case, we set $\frac{\partial
  f}{\partial x^i} \set g^i$. The weak derivatives of higher orders
are defined recursively.

Let $J\geq d$ be an integer and the functions
$\map{F^j,G}{\mathbb{R}^d}{\mathbb{R}}$ and
$\map{f^j,g^j,\alpha^j,\beta,\gamma^i}{[0,1]\times\mathbb{R}^d}{\mathbb{R}}$,
$j=1,\dots,J$, $i=1,\dots,d$, be such that for some constant $N>0$
\begin{enumerate}[label=(A\arabic{*}), ref=(A\arabic{*})]
  \setcounter{enumi}{\value{item}}
\item \label{item:2} The functions $F^j$ and $G$ are weakly
  differentiable, $G$ is strictly positive, the Jacobian matrix
  $\left(\frac{\partial F^j}{\partial x^i}\right)_{i=1,\dots,d, \;
    j=1,\dots,J}$ has rank $d$ almost surely under the Lebesgue
  measure on $\mathbb{R}^d$, and
  \begin{displaymath}
    \abs{\frac{\partial F^j}{\partial x^i}} + \abs{\frac{\partial G}{\partial
        x^i}} \leq e^{N(1 + \abs{x})}, \quad x\in \mathbb{R}^d. 
  \end{displaymath}
\item \label{item:3} The maps $t\mapsto e^{-N \abs{\cdot}}
  f^j(t,\cdot) \set \bigl(e^{-N \abs{x}} f^j(t,x)\bigr)_{x\in
    \mathbb{R}^d}$, $t\mapsto e^{-N \abs{\cdot}} g^j(t,\cdot)$ and
  $t\mapsto \alpha^j(t,\cdot)$, $t\mapsto \beta(t,\cdot)$, $t\mapsto
  \gamma^i(t,\cdot)$ of $[0,1]$ to $\mathbf{L}_{\infty}$ are analytic
  on $(0,1)$ and H\"older continuous on $[0,1]$.
  \setcounter{item}{\value{enumi}}
\end{enumerate}
Using these functions we define the random variable
\begin{displaymath}
  \xi \set G(X_1) e^{\int_0^1 \beta(t,X_t)dt},
\end{displaymath}
the equivalent probability measure $\widetilde{\mathbb{P}}$ and the
$\widetilde{\mathbb{P}}$-martingale $Y$ by
\begin{align*}
  \frac{d\widetilde{\mathbb{P}}}{d\mathbb{P}} &\set \exp\left(\int_0^1
    \gamma(s,X_s)dW_s -\frac12
    \int_0^1 \abs{\gamma(s,X_s)}^2 ds\right) \\
  Y_t &\set \widetilde{\mathbb{E}}[\xi|\mathcal{F}_t], \quad t\in
  [0,1],
\end{align*}
and the random variables
\begin{align*}
  \psi^j &\set F^j(X_1) e^{\int_0^1 \alpha^j(t,X_t)dt} + \int_0^1
  f^j(t,X_t) e^{\int_0^t\alpha^j(s,X_s)ds} dt \\
  &\quad + \int_0^1 \frac{g^j(t,X_t)}{Y_t}
  e^{\int_0^t(\alpha^j(s,X_s)+ \beta(s,X_s))ds} dt, \quad j=1,\dots,J.
\end{align*}

This is the main result of the paper.

\begin{Theorem}
  \label{th:1}
  Suppose that~\eqref{eq:4} and \ref{item:1}, \ref{item:2},
  and~\ref{item:3} hold.  Then the equivalent probability measure
  $\mathbb{Q}$ with the density under $\widetilde{\mathbb{P}}$
  \begin{displaymath}
    \frac{d\mathbb{Q}}{d\widetilde{\mathbb{P}}} \set
    \frac{\xi}{\widetilde{\mathbb{E}}[\xi]} = \frac{Y_1}{Y_0},
  \end{displaymath}
  and the $\mathbb{Q}$-martingale
  \begin{align*}
    S_t \set \mathbb{E}^{\mathbb{Q}}[\psi|\mathcal{F}_t], \quad t\in
    [0,1],
  \end{align*}
  with values in $\mathbb{R}^J$ are well-defined and every local
  martingale $M$ under $\mathbb{Q}$ is a stochastic integral with
  respect to $S$, that is, \eqref{eq:1} holds.
\end{Theorem}

The proof of Theorem~\ref{th:1} is given in
Section~\ref{sec:proof-theorem} and relies on the study of parabolic
equations in Section~\ref{sec:time-analyt-solut}.

\begin{Remark}
  \label{rem:3}
  The construction of $\mathbb{Q}$ and $S$ is done with a view to
  accommodate applications to financial economics,
  see~\cite{Kram:12b}, and to allow for a use of PDE techniques in the
  proof.  The $t$-analyticity condition on $f^j$ and $g^j$
  in~\ref{item:3} cannot be relaxed even if $X$ is a one-dimensional
  Brownian motion, see Example~\ref{ex:2} below.  By contrast, the
  $x$-regularity assumptions on the functions $F^j$, $G$, $f^j$, and
  $g^j$ in~\ref{item:2} and~\ref{item:3} admit weaker formulations
  where the $\mathbf{L}_\infty$ space can be replaced by appropriate
  $\mathbf{L}_p$ spaces. This generalization leads, however, to
  slightly more delicate and longer proofs and, probably, is not
  interesting for applications.
\end{Remark}

\begin{Remark}
  \label{rem:4}
  We stress that the $t$-analyticity conditions in~\ref{item:1}
  and~\ref{item:3} are stated for \emph{maps} of $(0,1)$ to
  $\mathbf{L}_\infty$ which is more than just boundedness and
  pointwise analyticity. For example, the function $h(t,x) \set
  \sin(tx)$, which is bounded and analytic on $\mathbb{R}\times
  \mathbb{R}$, does not define even a continuous map $t\mapsto
  h(t,\cdot)$ of $(0,1)$ to $\mathbf{L}_\infty$:
  \begin{displaymath}
    \norm{h(t,\cdot) - h(s,\cdot)}_{\mathbf{L}_\infty} \set \sup_{x\in
      \mathbb{R}}\abs{\sin(tx) - \sin(sx)} \geq 1, \quad 
    t\not=s. 
  \end{displaymath}
  The use of maps is essential for our proof of the theorem based on
  analytic semigroups.
\end{Remark}

We conclude with a few counter-examples illustrating the sharpness of
the conditions of the theorem. Our first two examples show that the
time analyticity assumptions on the volatility coefficient
$\sigma=\sigma(t,x)$ and on the functions $f^j = f^j(t,x)$ and $g^j =
g^j(t,x)$ cannot be relaxed. In both cases, we take $b^i(t,x) =
\alpha^j(t,x) = \beta(t,x) = \gamma^i(t,x) = 0$ and $G(x)=1$; in
particular, $\mathbb{Q}=\widetilde{\mathbb{P}} =\mathbb{P}$.

\begin{Example}
  \label{ex:1}
  We show that the assertion of Theorem~\ref{th:1} can fail to hold
  when all its conditions are satisfied except the $t$-analyticity of
  the volatility matrix $\sigma$. In our construction, $d=J=2$ and
  both $\sigma$ and its inverse $\sigma^{-1}$ are
  $\mathbf{C}^\infty$-matrices on $[0,1]\times \mathbb{R}^2$ which are
  bounded with all their derivatives and have analytic restrictions to
  $[0,\frac12)\times \mathbb{R}^2$ and $(\frac12,1]\times
  \mathbb{R}^2$. Moreover, the maps $t\mapsto \sigma^{ij}(t,\cdot)$ of
  $[0,1]$ to $\mathbf{C}(\mathbb{R}^2)$ belong to $\mathbf{C}^\infty$
  and have analytic restrictions to $[0,\frac12)$ and $(\frac12,1]$.

  Let $g=g(t)$ be a $\mathbf{C}^\infty$-function on $[0,1]$ which
  equals $0$ on $[0,\frac12]$ and is analytic and strictly positive on
  $(\frac{1}{2},1]$. Let $h = h(t,y)$ be an analytic function on
  $[0,1]\times \mathbb{R}$ such that $0\leq h\leq 1$, the function
  $h(1,\cdot)$ is non-constant, the map $t\mapsto h(t,\cdot)$ of
  $[0,1]$ to $\mathbf{C}$ is analytic, and
  \begin{displaymath}
    \frac{\partial h}{\partial t} + \frac12 \frac{\partial^2
      h}{\partial y^2} = 0.  
  \end{displaymath}
  For instance, we can take
  \begin{displaymath}
    h(t,y) = \frac12 (1 + e^{\frac{t-1}2}\sin y). 
  \end{displaymath}

  Define a $2$-dimensional diffusion $(X,Y)$ on $[0,1]$ by
  \begin{align*}
    X_t &\set \int_0^t \sqrt{1 + g(s)h(s,Y_s)} dB_s, \\
    Y_t &\set W_t,
  \end{align*}
  where $B$ and $W$ are independent Brownian motions.  Clearly, the
  volatility matrix
  \begin{displaymath}
    \sigma(t,x,y) = 
    \begin{pmatrix}
      \sqrt{1 + g(t)h(t,y)} & 0 \\
      0 & 1
    \end{pmatrix}
  \end{displaymath}
  has the announced properties and coincides with the identity matrix
  for $t\in [0,\frac12]$.

  Define the functions $F = F(x,y)$ and $H=H(x,y)$ on $\mathbb{R}^2$
  as
  \begin{align*}
    F(x,y) &\set x, \\
    H(x,y) &\set x^2 - 1 - h(1,y) \int_0^1 g(t)dt.
  \end{align*}
  As $h(1,\cdot)$ is non-constant and analytic, the set of zeros for
  $\frac{\partial h}{\partial y}(1,\cdot)$ is at most countable.
  Since the determinant of the Jacobian matrix for $(F,H)$ is given by
  \begin{displaymath}
    \frac{\partial F}{\partial x}\frac{\partial H}{\partial y} - 
    \frac{\partial F}{\partial y}\frac{\partial H}{\partial x} 
    =  - \frac{\partial h}{\partial y}(1,y)\int_0^1 g(t)dt,
  \end{displaymath}
  it follows that this Jacobian matrix has full rank almost surely.
  
  Observe now that
  \begin{align*}
    S_t & \set   \mathbb{E}[F(X_1,Y_1)|\mathcal{F}_t] =  X_t, \\
    R_t & \set \mathbb{E}[H(X_1,Y_1)|\mathcal{F}_t] = X^2_t - t -
    h(t,Y_t) \int_0^t g(s)ds,
  \end{align*}
  which can be verified by Ito's formula. As $g(t)=0$ for $t\in
  [0,\frac12]$, it follows that $S_t = B_t$ and $R_t = B^2_t - t$ on
  $[0,\frac12]$. Hence, the Brownian motion $Y=W$ cannot be written as
  a stochastic integral with respect to $(S,R)$.
\end{Example}

\begin{Example}
  \label{ex:2}
  This counter-example shows the necessity of the $t$-analyticity
  assumption on $f^j=f^j(t,x)$ and $g^j = g^j(t,x)$ in~\ref{item:3}.
  Let $h=h(t)$ be a $\mathbf{C}^\infty$-function on $[0,1]$ which
  equals $0$ on $[0,\frac12]$, is analytic on $(\frac12,1]$, and is
  such that $h(1)\not = 0$. For the functions
  \begin{align*}
    F(x) & \set e^{h(1)x}, \\
    f(t,x) & \set -\bigl(h'(t)x + \frac12 h^2(t)\bigr) e^{h(t)x},
  \end{align*}
  the conditions \ref{item:2} and \ref{item:3} hold except the time
  analyticity of the map $t\to e^{-N \abs{\cdot}}f(t,\cdot)$ of
  $[0,1]$ to $\mathbf{L}_{\infty}$. This map belongs instead to
  $\mathbf{C}^\infty$ and has analytic restrictions to $[0,\frac12)$
  and $(\frac12,1]$.

  Take $X$ to be a one-dimensional Brownian motion:
  \begin{displaymath}
    X_t \set W_t, \quad t\in [0,1],  
  \end{displaymath} 
  and observe that, by Ito's formula,
  \begin{displaymath}
    S_t \set  \mathbb{E}[\psi|\mathcal{F}_t] =  e^{h(t)W_t} -
    \int_0^t\bigl(h'(s)W_s +  
    \frac12 h^2(s)\bigr) e^{h(s)W_s} ds, 
  \end{displaymath}
  where
  \begin{displaymath}
    \psi = F(X_1) + \int_0^1 f(t,X_t)dt.
  \end{displaymath}
  For $t\in [0,\frac12]$ we have $h(t)=0$ and, therefore, $S_t =
  1$. Hence, a local martingale $M$ which is non-constant on
  $[0,\frac12]$ cannot be a stochastic integral with respect to $S$.
\end{Example}

When the diffusion coefficients $\sigma^{ij}$ and $b^i$ and the
functions $f^j$, $g^j$, $\alpha^j$, $\beta$, and $\gamma^i$
in~\ref{item:3} are also analytic with respect to the state variable
$x$, the results in \cite{HugMalTrub:12} and \cite{RiedHerz:12} show
that in \ref{item:2} it is sufficient to require the Jacobian matrix
of $F = F(x)$ to have rank $d$ only on an open set. The following
example shows that in the case of $\mathbf{C}^\infty$ functions this
simplification is not possible anymore.

\begin{Example}
  \label{ex:3}
  Let $d=J = 2$ and let $\map{h}{\mathbb{R}}{\mathbb{R}}$ be a
  $\mathbf{C}^\infty$ function such that $h(x)= 0$ for $x \leq 0$,
  while $h^\prime(x)>0$ and $h''(x)$ is bounded for $x>0$. For
  instance, take $h(x) = \ind{x>0} e^{-1/x}$.

  Define the diffusion processes $X$ and $Y$ on $[0,1]$ by
  \begin{align*}
    X_t & = B_t,\\
    Y_t & = \int_0^t h^{\prime\prime}(X_s)ds + W_t,
  \end{align*}
  where $B$ and $W$ are independent Brownian motions. Clearly, the
  diffusion coefficients of $(X,Y)$ satisfy~\ref{item:1}.

  Define the functions $F = F(x,y)$ and $H=H(x,y)$ on $\mathbb{R}^2$
  as
  \begin{align*}
    F(x,y) & = y,\\
    H(x,y) & = y - 2h(x),
  \end{align*}
  and the function $f=f(t,x,y)$ on $[0,1]\times \mathbb{R}^2$ as
  \begin{displaymath}
    f(t,x,y) = -h^{\prime\prime}(x).
  \end{displaymath}
  Observe that the determinant of the Jacobian matrix for $(F,H)$ is
  given by
  \begin{displaymath}
    \frac{\partial F}{\partial x}\frac{\partial H}{\partial y} - 
    \frac{\partial F}{\partial y}\frac{\partial H}{\partial x} 
    =  2 h^{\prime}(x),
  \end{displaymath}
  and, hence, this Jacobian matrix has full rank on the set
  $(0,\infty)\times \mathbb{R}$.

  A simple application of Ito's formula yields
  \begin{align*}
    S_t &\set \mathbb{E}[F(X_1,Y_1) + \int_0^1 f(s,X_s,Y_s)ds
    |\mathcal{F}_t] = W_t, \\
    R_t &\set \mathbb{E}[H(X_1,Y_1) |\mathcal{F}_t]= W_t - 2\int_0^t
    h'(X_s) dB_s.
  \end{align*}
  Hence, any martingale in the form
  \begin{equation*}
    M_t = \int_0^t \phi(X_s) dB_s,
  \end{equation*}
  where the function $\phi=\phi(x)$ is different from zero for $x\leq
  0$, cannot be written as a stochastic integral with respect to $(S,
  R)$.
\end{Example}

\section{A time analytic solution of a parabolic equation}
\label{sec:time-analyt-solut}

The proof of Theorem~\ref{th:1} relies on the study of a parabolic
equation in Theorem~\ref{th:3} below.

For reader's convenience, recall the definition of the classical
Sobolev spaces $\mathbf{W}^{m}_p$ on $\mathbb{R}^d$ where $m\in
\braces{0,1,\dots}$ and $p\geq 1$.  When $m=0$ we get the classical
Lebesgue spaces $\mathbf{L}_p=\mathbf{L}_p(\mathbb{R}^d,dx)$ of
real-valued functions $f$ on $\mathbb{R}^d$ with the norm
\begin{displaymath}
  \norm{f}_{\mathbf{L}_p} \set \left(\int_{\mathbb{R}^d}
    \abs{f(x)}^p dx\right)^{\frac{1}{p}}.
\end{displaymath}
When $m\in \braces{1,\dots}$ the Sobolev space $\mathbf{W}^{m}_p$
consists of all $m$-times weakly differentiable functions $f$ such
that
\begin{displaymath}
  \norm{f}_{\mathbf{W}^{m}_p} \set \norm{f}_{\mathbf{L}_{p}} +
  \sum_{1\leq \abs{\alpha} \leq m}
  \norm{D^{\alpha}f}_{\mathbf{L}_{p}} <\infty 
\end{displaymath}
and is a Banach space with such a norm. The summation is taken with
respect to multi-indexes $\alpha = (\alpha_1,\dots,\alpha_d)$ of
non-negative integers, $\abs{\alpha} \set \sum_{i=1}^d \alpha_i$ and
\begin{displaymath}
  D^\alpha \set 
  \frac{\partial^{|\alpha|}}{\partial x_1^{\alpha_1}\dots \partial 
    x_d^{\alpha_d}}. 
\end{displaymath}

For $t\in [0,1]$ and $x\in\mathbb{R}^d$ consider an elliptic operator
\begin{displaymath}
  A(t) \set \sum_{i,j=1}^d a^{ij}(t,x) \frac{\partial^2}{\partial
    x^i \partial x^j} + \sum_{i=1}^d b^i(t,x) \frac{\partial}{\partial  
    x^i} + c(t,x),
\end{displaymath}
where $a^{ij}$, $b^i$, and $c$ are measurable functions on
$[0,1]\times \mathbb{R}^d$ such that
\begin{enumerate}[label=(B\arabic{*}), ref=(B\arabic{*})]
\item \label{item:4} the maps $t\mapsto a^{ij}(t,\cdot)$ of $[0,1]$ to
  $\mathbf{C}$ and $t\mapsto b^{i}(t,\cdot)$ and $t\mapsto c(t,\cdot)$
  of $[0,1]$ to $\mathbf{L}_\infty$ are analytic on $(0,1)$ and
  H\"older continuous on $[0,1]$. The matrix $a$ is symmetric: $a^{ij}
  = a^{ji}$, uniformly elliptic: there exists $N>0$ such that
  \begin{displaymath}
    \ip{y}{a(t,x)y}  \geq \frac1{N^2} \abs{y}^2, \quad
    (t,x)\in [0,1]\times \mathbb{R}^d, \quad y\in \mathbb{R}^d,
  \end{displaymath}
  and uniformly continuous with respect to $x$: there exists a
  decreasing function $\omega = (\omega(\epsilon))_{\epsilon>0}$ such
  that $\omega(\epsilon)\to 0$ as $\epsilon\downarrow 0$ and for all
  $t\in [0,1]$ and $y,z\in \mathbb{R}^d$
  \begin{displaymath}
    \abs{a^{ij}(t,y) - a^{ij}(t,z)} \leq
    \omega(\abs{y-z}). 
  \end{displaymath}
  \setcounter{item}{\value{enumi}}
\end{enumerate}
Let $g=\map{g(x)}{\mathbb{R}^d}{\mathbb{R}}$ and
$f=\map{f(t,x)}{[0,1]\times \mathbb{R}^d}{\mathbb{R}}$ be measurable
functions such that for some $p>1$
\begin{enumerate}[label=(B\arabic{*}), ref=(B\arabic{*})]
  \setcounter{enumi}{\value{item}}
\item \label{item:5} the function $g$ belongs to $\mathbf{W}^1_p$ and
  the map $t\mapsto f(t,\cdot)$ from $[0,1]$ to $\mathbf{L}_{p}$ is
  analytic on $(0,1)$ and H\"older continuous on $[0,1]$.
  \setcounter{item}{\value{enumi}}
\end{enumerate}

\begin{Theorem}
  \label{th:2}
  Let $p>1$ and suppose the conditions~\ref{item:4} and~\ref{item:5}
  hold. Then there exists a unique measurable function $u=u(t,x)$ on
  $[0,1]\times \mathbb{R}^d$ such that
  \begin{enumerate}
  \item \label{item:6} $t\mapsto u(t,\cdot)$ is a H\"older continuous
    map of $[0,1]$ to $\mathbf{L}_p$ whose restriction on $(0,1]$ is
    continuously differentiable,
  \item \label{item:7} $t\mapsto u(t,\cdot)$ is a continuous map of
    $[0,1]$ to $\mathbf{W}^{1}_p$,
  \item \label{item:8} $t\mapsto u(t,\cdot)$ is a continuous map of
    $(0,1]$ to $\mathbf{W}^{2}_p$ whose restriction on $(0,1)$ is
    analytic,
  \end{enumerate}
  and such that $u=u(t,x)$ solves the parabolic equation:
  \begin{align}
    \label{eq:6}
    \frac{\partial u}{\partial t} &= A(t)u + f, \quad t\in (0,1], \\
    \label{eq:7}
    u(0,\cdot) &= g.
  \end{align}
\end{Theorem}

The proof relies on results from the theory of analytic semigroups
where our main references are \citet{Pazy:83} and \citet{Yagi:10}. We
first introduce some notations and state a few lemmas.

Let $\mathbf{X}$ and $\mathbf{D}$ be Banach spaces.  By
$\mathcal{L}(\mathbf{X},\mathbf{D})$ we denote the Banach space of
bounded linear operators $\map{T}{\mathbf{X}}{\mathbf{D}}$ endowed
with the operator norm. A shorter notation $\mathcal{L}(\mathbf{X})$
is used for $\mathcal{L}(\mathbf{X},\mathbf{X})$. We shall write
$\mathbf{D}\subset \mathbf{X}$ if $\mathbf{D}$ is \emph{continuously
  embedded} into $\mathbf{X}$, that is, the elements of $\mathbf{D}$
form a subset of $\mathbf{X}$ and there is a constant $N>0$ such that
$\norm{x}_{\mathbf{X}} \leq N\norm{x}_{\mathbf{D}}$, $x\in
\mathbf{D}$. We shall write $\mathbf{D}= \mathbf{X}$ if
$\mathbf{D}\subset \mathbf{X}$ and $\mathbf{X}\subset \mathbf{D}$.

Let $\mathbf{D}\subset \mathbf{X}$. A Banach space $\mathbf{E}$ is
called an \emph{interpolation space} between $\mathbf{D}$ and
$\mathbf{X}$ if $\mathbf{D}\subset \mathbf{E}\subset \mathbf{X}$ and
any linear operator $T\in \mathcal{L}(\mathbf{X})$ whose restriction
to $\mathbf{D}$ belongs to $\mathcal{L}(\mathbf{D})$ also has its
restriction to $\mathbf{E}$ in $\mathcal{L}(\mathbf{E})$; see
\citet[Section 2.4]{BergLofs:76}.

The following lemma is used in the proof of item~\ref{item:7} of the
theorem.

\begin{Lemma}
  \label{lem:1}
  Let $\mathbf{D}$, $\mathbf{E}$, and $\mathbf{X}$ be Banach spaces
  such that $\mathbf{D}\subset \mathbf{X}$, $\mathbf{E}$ is an
  interpolation space between $\mathbf{D}$ and $\mathbf{X}$, and
  $\mathbf{D}$ is dense in $\mathbf{E}$. Let $(T_n)_{n\geq 1}$ be a
  sequence of linear operators in $\mathcal{L}(\mathbf{X})$ such that
  $\lim_{n\to \infty} \norm{T_n x}_{\mathbf{X}} = 0$ for every $x\in
  \mathbf{X}$ and $\lim_{n\to \infty} \norm{T_n x}_{\mathbf{D}} = 0$
  for every $x\in \mathbf{D}$.  Then $\lim_{n\to \infty} \norm{T_n
    x}_{\mathbf{E}} = 0$ for every $x\in \mathbf{E}$.
\end{Lemma}
\begin{proof}
  The uniform boundedness theorem implies that the sequence
  $(T_n)_{n\geq 1}$ is bounded both in $\mathcal{L}(\mathbf{X})$ and
  $\mathcal{L}(\mathbf{D})$.  Due to the Banach property, $\mathbf{E}$
  is a uniform interpolation space between $\mathbf{D}$ and
  $\mathbf{X}$, that is, there is a constant $M>0$ such that
  \begin{displaymath}
    \norm{T}_{\mathcal{L}(\mathbf{E})} \leq M
    \max(\norm{T}_{\mathcal{L}(\mathbf{X})},
    \norm{T}_{\mathcal{L}(\mathbf{D})}) \mtext{for any} T\in
    \mathcal{L}(\mathbf{X}) \cap \mathcal{L}(\mathbf{D});
  \end{displaymath}
  see Theorem 2.4.2 in \cite{BergLofs:76}. Hence, $(T_n)_{n\geq 1}$ is
  also bounded in $\mathcal{L}(\mathbf{E})$. The density of
  $\mathbf{D}$ in $\mathbf{E}$ then yields the result.
\end{proof}

Let $A$ be an (unbounded) closed linear operator on $\mathbf{X}$. We
denote by $\mathbf{D}(A)$ the domain of $A$ and assume that it is
endowed with the graph norm of $A$:
\begin{displaymath}
  \norm{x}_{\mathbf{D}(A)} \set \norm{Ax}_{\mathbf{X}} +
  \norm{x}_{\mathbf{X}}.
\end{displaymath}
Then $\mathbf{D}(A)$ is a Banach space.  Recall that the {resolvent
  set} $\rho(A)$ of $A$ consists of complex numbers $\lambda$ for
which the operator $\map{\lambda I - A}{\mathbf{D}(A)}{\mathbf{X}}$,
where $I$ is the identity operator, is invertible; the inverse
operator is called the resolvent and is denoted by $R(\lambda,A)$. The
bounded inverse theorem implies that $R(\lambda,A)\in
\mathcal{L}(\mathbf{X},\mathbf{D}(A))$ and, in particular,
$R(\lambda,A)\in \mathcal{L}(\mathbf{X})$.

The operator $A$ is called \emph{sectorial} if there are constants $M
> 0$, $r\in\mathbb{R}$, and $\theta\in \left(0,\frac{\pi}{2}\right)$
such that the sector
\begin{equation}
  \label{eq:8}
  S_{r,\theta}\set \descr{\lambda\in\mathbb{C}}{\lambda \neq r
    \mtext{and} \abs{\arg(\lambda - r)} \leq \pi - \theta}
\end{equation}
of the complex plane $\mathbb{C}$ is a subset of $\rho(A)$ and
\begin{equation}
  \label{eq:9}
  \norm{R(\lambda, A)}_{\mathcal{L}(\mathbf{X})}
  \leq \frac{M}{1 + \abs{\lambda}},\quad \lambda\in S_{r,\theta}.  
\end{equation}
The set of such sectorial operators will be denoted by
$\mathcal{S}(M,r,\theta)$. Sectorial operators whose domains are dense
in $\mathbf{X}$ coincide with the generators of analytic semi-groups,
see \citet[Section 2.5]{Pazy:83}.

The following two lemmas are used in the proof of item~\ref{item:8} of
the theorem.

\begin{Lemma}
  \label{lem:2}
  Let $\mathbf{X}$ and $\mathbf{D}$ be Banach spaces such that
  $\mathbf{D}\subset \mathbf{X}$ and let $A = (A(t))_{t\in [0,1]}$ be
  closed linear operators on $\mathbf{X}$ such that $\mathbf{D}(A(t))
  = \mathbf{D}$ for all $t\in[0,1]$. Suppose
  $\map{A}{[0,1]}{\mathcal{L}(\mathbf{D},\mathbf{X})}$ is a continuous
  map, and there are $M > 0$, $r < 0$, and
  $\theta\in\left(0,\frac{\pi}{2}\right)$ such that
  $A(t)\in\mathcal{S}(M, r, \theta)$ for all $t\in[0,1]$.

  Then for every $\lambda \in S_{r,\theta}$ the map $t\mapsto
  R(\lambda,A(t))$ of $[0,1]$ to $\mathcal{L}(\mathbf{X},\mathbf{D})$
  is continuous and there is $N>0$ such that
  \begin{equation}
    \label{eq:10}
    \norm{R(\lambda,A(t))}_{\mathcal{L}(\mathbf{X}, \mathbf{D})} \leq
    N, \quad \lambda\in  S_{r,\theta}, \; t\in [0,1]. 
  \end{equation}
\end{Lemma}
\begin{proof}
  If $A\in \mathcal{S}(M,r,\theta)$, then for $\lambda \in
  S_{r,\theta}$
  \begin{equation}
    \label{eq:11}
    \norm{R(\lambda,A)}_{\mathcal{L}(\mathbf{X},\mathbf{D}(A))} \leq
    \norm{R(\lambda,A)}_{\mathcal{L}(\mathbf{X})} +
    \norm{AR(\lambda,A)}_{\mathcal{L}(\mathbf{X})} \leq M + 1,
  \end{equation}
  where we used~\eqref{eq:9} and the identity $AR(\lambda,A) = \lambda
  R(\lambda,A) - I$. As
  $\map{A}{[0,1]}{\mathcal{L}(\mathbf{D},\mathbf{X})}$ is a continuous
  function and $\mathbf{D}(A(t)) = \mathbf{D}$, the Banach spaces
  $\mathbf{D}$ and $\mathbf{D}(A(t))$, $t\in [0,1]$, are uniformly
  equivalent, that is, there is $L>0$ such that
  \begin{equation}
    \label{eq:12}
    \frac1L \norm{x}_{\mathbf{D}(A(t))} \leq \norm{x}_{\mathbf{D}}
    \leq L \norm{x}_{\mathbf{D}(A(t))}, \quad t\in [0,1],\; x\in
    \mathbf{D}. 
  \end{equation}
  From~\eqref{eq:11} and~\eqref{eq:12} we obtain~\eqref{eq:10} with
  $N=L(M+1)$.  As, for $s,t\in [0,1]$ and $\lambda \in S_{r,\theta}$,
  \begin{displaymath}
    R(\lambda,A(t)) - R(\lambda,A(s)) =
    R(\lambda,A(t))(A(t)-A(s))R(\lambda,A(s)), 
  \end{displaymath}
  we deduce from~\eqref{eq:10} that
  \begin{displaymath}
    \norm{R(\lambda,A(t)) -
      R(\lambda,A(s))}_{\mathcal{L}(\mathbf{X},\mathbf{D})} 
    \leq
    N^2 \norm{(A(t)-A(s))}_{\mathcal{L}(\mathbf{D},\mathbf{X})}.
  \end{displaymath}
  The desired continuity of $t\mapsto R(\lambda,A(t))$ in
  $\mathcal{L}(\mathbf{X},\mathbf{D})$ follows now from the continuity
  of $t\mapsto {A}(t)$ in $\mathcal{L}(\mathbf{D},\mathbf{X})$.
\end{proof}

\begin{Lemma}
  \label{lem:3}
  Let $\mathbf{X}$ and $\mathbf{D}$ be Banach spaces such that
  $\mathbf{D}\subset \mathbf{X}$ and let $A = (A(t))_{t\in [0,1]}$ be
  closed linear operators on $\mathbf{X}$ such that $\mathbf{D}(A(t))
  = \mathbf{D}$ for all $t\in[0,1]$. Suppose
  $\map{A}{[0,1]}{\mathcal{L}(\mathbf{D},\mathbf{X})}$ is an analytic
  map, and there are $M > 0$, $r < 0$, and
  $\theta\in\left(0,\frac{\pi}{2}\right)$ such that
  $A(t)\in\mathcal{S}(M, r, \theta)$ for all $t\in[0,1]$.

  Then there exist a convex open set $U$ in $\mathbb{C}$ containing
  $[0,1]$ and an analytic extension of $A$ to $U$ such that
  $A(z)\in\mathcal{S}(2M, r,\theta)$ for all $z\in U$ and the function
  $\map{A^{-1}}{[0,1]}{\mathcal{L}(\mathbf{X},\mathbf{D})}$ is
  analytic.
\end{Lemma}
\begin{proof}
  Since $r<0$, the operator $A(t)$ is invertible for every $t\in
  [0,1]$. As $\map{A}{[0,1]}{\mathcal{L}(\mathbf{D},\mathbf{X})}$ is
  analytic, the inverse function
  $\map{B=A^{-1}}{[0,1]}{\mathcal{L}(\mathbf{X},\mathbf{D})}$ is
  well-defined and analytic. Clearly, there is an open convex set $U$
  in $\mathbb{C}$ containing $[0,1]$ on which both $A$ and $B$ can be
  analytically extended. Then $B = A^{-1}$ on $U$, as $AB$ is an
  analytic function on $U$ with values in $\mathcal{L}(\mathbf{X})$
  which on $[0,1]$ equals the identity operator. By Lemma~\ref{lem:2}
  there is a constant $N>0$ such that the inequality~\eqref{eq:10}
  holds. Of course, we can choose $U$ so that for every $z\in U$ there
  is $t\in [0,1]$ such that
  \begin{equation}
    \label{eq:13}
    \norm{A(z) - A(t)}_{\mathcal{L}(\mathbf{D},\mathbf{X})} \leq
    \frac1{2N}.
  \end{equation}

  Fix $\lambda \in S_{r,\theta}$ and take $t\in [0,1]$ and $z\in U$
  satisfying~\eqref{eq:13}. Then
  \begin{displaymath}
    \norm{(A(z) - A(t))R(\lambda,A(t))}_{\mathcal{L}(\mathbf{X})} \leq
    \frac1{2} 
  \end{displaymath}
  and, hence, the operator $I - (A(z) - A(t))R(\lambda,A(t))$ in
  $\mathcal{L}(\mathbf{X})$ is invertible and the norm of its inverse
  is not greater than $2$. Since
  \begin{displaymath}
    \lambda I - A(z) = (I -  (A(z) - A(t))R(\lambda,A(t)))(\lambda I - A(t)),
  \end{displaymath}
  we deduce that the resolvent $R(\lambda,A(z))$ is well-defined and
  \begin{displaymath}
    \norm{R(\lambda,A(z))}_{\mathcal{L}(\mathbf{X})} \leq
    \frac{2M}{1+\abs{\lambda}}.  
  \end{displaymath}
  This completes the proof.
\end{proof}

\begin{proof}[Proof of Theorem~\ref{th:2}]
  Under \ref{item:4} for every $t\in [0,1]$ the operator $A(t)$ is
  closed in $\mathbf{L}_p$ and has $\mathbf{W}^2_p$ as its domain:
  \begin{equation}
    \label{eq:14}
    \mathbf{D}(A(t)) = \mathbf{W}^2_p. 
  \end{equation}
  Moreover, the operators $(A(t))_{t\in [0,1]}$ are \emph{sectorial}
  with the same constants $M > 0$, $r\in\mathbb{R}$, and $\theta\in
  \left(0,\frac{\pi}{2}\right)$:
  \begin{equation}
    \label{eq:15}
    A(t) \in \mathcal{S}(M,r,\theta), \quad t\in [0,1]. 
  \end{equation}
  These results can be found, for example, in \citet{Kryl:08}, see
  Section 13.4 and Exercise 13.5.1.

  To insure the existence of fractional powers for the operators
  $-A(t)$ it is convenient for us to assume that the sector
  $S_{r,\theta}$ defined in~\eqref{eq:8} contains $0$ or,
  equivalently, that $r<0$.  This condition does not restrict any
  generality as for $s\in \mathbb{R}$ the substitution $u(t,x) \to
  e^{st}u(t,x)$ in~\eqref{eq:6} corresponds to the shift $A(t) \to
  A(t) + s$ in the operators $A(t)$.

  From~\ref{item:4} we deduce the existence of $M>0$ and $\delta>0$
  such that for $v\in \mathbf{W}^2_p$
  \begin{equation}
    \label{eq:16}
    \norm{(A(t) - A(s))v}_{\mathbf{L}_p} \leq M \abs{t-s}^\delta
    \norm{v}_{\mathbf{W}^2_p}, \quad s,t\in [0,1]. 
  \end{equation}
  Conditions~\eqref{eq:14}, \eqref{eq:15}, and \eqref{eq:16} for the
  operators $A=A(t)$, the H\"older continuity of $f=f(t)$
  in~\ref{item:5}, and the fact that $g\in \mathbf{L}_p$ imply the
  existence and uniqueness of the {classical} solution $u=u(t,x)$ to
  the initial value problem \eqref{eq:6}--\eqref{eq:7} in
  $\mathbf{L}_p$; see Theorem 7.1 in Section 5.7 of \cite{Pazy:83} or
  Theorem~3.9 in Section 3.6 of \cite{Yagi:10}. Recall that $u=u(t,x)$
  is the \emph{classical solution} to \eqref{eq:6} and \eqref{eq:7} if
  $u(t,\cdot)\in \mathbf{W}^2_p$ for $t\in (0,1]$, the map $t\mapsto
  u(t,\cdot)$ of $[0,1]$ to $\mathbf{L}_p$ is continuous, the
  restriction of this map to $(0,1]$ is continuously differentiable,
  and the equations \eqref{eq:6} and \eqref{eq:7} hold.

  The continuity of the map $t\mapsto u(t,\cdot)$ of $(0,1]$ to
  $\mathbf{W}^2_p$ follows from the continuity of the map $t\mapsto
  A(t)u(t,\cdot) = \frac{\partial u}{\partial t}(t,\cdot) -
  f(t,\cdot)$ of $(0,1]$ to $\mathbf{L}_p$ and the continuity of the
  map $t\mapsto {A^{-1}(t)}$ of $[0,1]$ to
  $\mathcal{L}(\mathbf{L}_p,\mathbf{W}^2_p)$, which holds by
  Lemma~\ref{lem:2}.

  To verify item~\ref{item:6} we still need to check the H\"older
  continuity of the map $t\mapsto u(t,\cdot)$ of $[0,1]$ to
  $\mathbf{L}_p$. We use Theorem~3.10 in Section~3.8.2 of
  \citet{Yagi:10} dealing with {maximal regularity} of solutions to
  evolution equations. This theorem implies the existence of constants
  $\delta>0$ and $M>0$ such that
  \begin{equation}
    \label{eq:17}
    \norm{\frac{\partial u}{\partial t}(t,\cdot)}_{\mathbf{L}_p} \leq M
    t^{\delta-1}, \quad t\in (0,1], 
  \end{equation}
  provided that the operators $A=A(t)$ satisfy \eqref{eq:14},
  \eqref{eq:15}, and~\eqref{eq:16}, the function $f$ is H\"older
  continuous as in~\ref{item:5}, and
  \begin{equation}
    \label{eq:18}
    g \in \mathbf{D}((-A(0))^\gamma) \quad \text{for some} \quad \gamma>0, 
  \end{equation}
  where $\mathbf{D}((-A(0))^\gamma)$ is the domain of the fractional
  power $\gamma$ of the operator $- A(0)$ acting in
  $\mathbf{L}_p$. The inequality~\eqref{eq:17} clearly implies the
  H\"older continuity of $\map{u(t,\cdot)}{[0,1]}{\mathbf{L}_p}$ and,
  hence, to complete the proof of item~\ref{item:6} we only need to
  verify~\eqref{eq:18}.

  Since $g\in \mathbf{W}^1_p$, we obtain~\eqref{eq:18} if
  \begin{equation}
    \label{eq:19}
    \mathbf{W}^1_p \subset \mathbf{D}((- A(0))^\gamma), \quad
    \gamma \in (0,\frac12). 
  \end{equation}
  Denote by $\Delta \set \sum_{i} \frac{\partial^2}{\partial x_i^2}$
  the Laplace operator and recall that $1-\Delta$ is a sectorial
  operator and
  \begin{displaymath}
    \mathbf{W}^1_p = \mathbf{D}((1-\Delta)^{1/2}), \quad
    \mathbf{D}(-A(0)) = \mathbf{W}^2_p = \mathbf{D}((1-\Delta)),
  \end{displaymath}
  see e.g. \cite[Theorem 13.3.12]{Kryl:08}.  The
  embedding~\eqref{eq:19} now follows from the fact that for constants
  $0<\alpha<\beta<1$ and sectorial operators $A$ and $B$ such that
  $\mathbf{D}(B)\subset \mathbf{D}(A)$ and such that the fractional
  powers $(-A)^\alpha$ and $(-B)^\beta$ are well-defined we have
  $\mathbf{D}((-B)^{\beta}) \subset \mathbf{D}((-A)^{\alpha})$, see
  \cite[Theorem 2.25]{Yagi:10}. This finishes the proof of
  item~\ref{item:6}.

  Another consequence of the maximal regularity properties of $u$
  given in \cite[Theorem 3.10]{Yagi:10} is that the map
  $\map{u(t,\cdot)}{[0,1]}{\mathbf{W}^2_p}$ is continuous if $g \in
  \mathbf{W}^2_p = \mathbf{D}(A(0))$. We shall apply this result
  shortly to prove item~\ref{item:7}.

  For $t\in [0,1]$ define a linear operator $T(t)$ on $\mathbf{L}_p$
  such that for $h\in \mathbf{L}_p$ the function $v = v(t,x)$ given by
  $v(t,\cdot) = T(t)h$ is the unique classical solution in
  $\mathbf{L}_p$ of the homogeneous problem:
  \begin{equation}
    \label{eq:20}
    \frac{\partial v}{\partial t} = A(t)v, \quad v(0,\cdot) = h.
  \end{equation}
  Actually, $T(t) = U(t,0)$, where $U = (U(t,s))_{0\leq s\leq t\leq
    1}$ is the evolution system for $A = A(t)$; see \citet[Chapter
  5]{Pazy:83}, but we shall not use this relation. Of course, the
  properties established above for $u = u(t,x)$ also hold for the
  solution $v=v(t,x)$ to~\eqref{eq:20}. It follows that for $h\in
  \mathbf{L}_p$ the map $t\mapsto T(t)h$ of $[0,1]$ to $\mathbf{L}_p$
  is well-defined and continuous and if $h\in \mathbf{W}^2_p$ then the
  same map is also continuous in $\mathbf{W}^2_p$. Recall now that
  $\mathbf{W}^1_p$ is an interpolation space between $\mathbf{L}^p$
  and $\mathbf{W}^2_p$, more precisely, a midpoint in complex
  interpolation, see \cite[Theorem 6.4.5]{BergLofs:76}. Since
  $\mathbf{W}^2_p$ is dense in $\mathbf{W}^1_p$, Lemma~\ref{lem:1}
  yields the continuity of the map $t\mapsto T(t)h$ of $[0,1]$ to
  ${\mathbf{W}^1_p}$.

  Observe now that $u=u(t,x)$ can be decomposed as
  \begin{displaymath}
    u(t,\cdot) = T(t)g + w(t,\cdot), 
  \end{displaymath}
  where $w(t,\cdot)$ is the unique classical solution in
  $\mathbf{L}_p$ of the inhomogeneous problem:
  \begin{displaymath}
    \frac{\partial w}{\partial t} = A(t)w + f, \quad w(0,\cdot) = 0. 
  \end{displaymath}
  Since $w$ coincides with $u$ in the special case $g=0$, the map
  $t\mapsto w(t,\cdot)$ is continuous in $\mathbf{W}^2_p$ and, hence,
  also continuous in $\mathbf{W}^1_p$.  This completes the proof of
  item~\ref{item:7}.

  In item~\ref{item:8} it only remains to verify the analyticity of
  the map $t\mapsto u(t,\cdot)$ of $(0,1)$ to $\mathbf{W}^2_p$. Fix
  $0<\epsilon<1/2$. The condition~\ref{item:4} implies the analyticity
  of the function
  $\map{A=A(t)}{[\epsilon,1-\epsilon]}{\mathcal{L}(\mathbf{W}^2_p,\mathbf{L}_p)}$. Let
  $U$ be an open convex set in $\mathbb{C}$ containing
  $[\epsilon,1-\epsilon]$ on which there is an analytic extension of
  $A$ satisfying the assertions of Lemma~\ref{lem:3}. We choose $U$ so
  that $\map{f = f(t,\cdot)}{[\epsilon,1-\epsilon]}{\mathbf{L}_p}$ can
  also be analytically extended on $U$. Theorem 2 in
  \citet{KatoTanabe:67} now implies the analyticity of the map
  $t\mapsto u(t,\cdot)$ of $[\epsilon,1-\epsilon]$ to $\mathbf{L}_p$.
  However, as
  \begin{displaymath}
    u(t,\cdot) = \bigl(A(t)\bigr)^{-1} (\frac{\partial u}{\partial t} -
    f(t,\cdot)), 
  \end{displaymath}
  and since, by Lemma~\ref{lem:3}, the
  $\mathcal{L}(\mathbf{L}_p,\mathbf{W}^2_p)$-valued function
  $\bigl(A(t)\bigr)^{-1}$ on $[\epsilon,1-\epsilon]$ is analytic, the
  map $t\mapsto u(t,\cdot)$ on $[\epsilon,1-\epsilon]$ is also
  analytic in $\mathbf{W}^2_p$. As $\epsilon>0$ is any small number we
  obtain the result.
\end{proof}

In the proof of our main Theorem~\ref{th:1} we actually use
Theorem~\ref{th:3} below, which is a corollary of Theorem~\ref{th:2}.
Instead of~\ref{item:5} we assume that the measurable functions
$g=\map{g(x)}{\mathbb{R}^d}{\mathbb{R}}$ and
$f=\map{f(t,x)}{[0,1]\times \mathbb{R}^d}{\mathbb{R}}$ have the
following properties:
\begin{enumerate}[label=(B\arabic{*}), ref=(B\arabic{*})]
  \setcounter{enumi}{\value{item}}
\item \label{item:9} There is a constant $N\geq 0$ such that
  $e^{-N\abs{\cdot}}\frac{\partial g}{\partial
    x^i}(\cdot)\in\mathbf{L}_{\infty}$ and for every $p \geq 1$ the
  map $t\mapsto e^{-N\abs{\cdot}} f(t,\cdot)$ from $[0,1]$ to
  $\mathbf{L}_{p}$ is analytic on $(0,1)$ and H\"older continuous on
  $[0,1]$. \setcounter{item}{\value{enumi}}
\end{enumerate}
Fix a function $\phi = \phi(x)$ such that
\begin{equation}
  \label{eq:21}
  \text{$\phi \in \mathbf{C}^\infty(\mathbb{R}^d)$ and $\phi(x) = \abs{x}$
    when $\abs{x}\geq 1$.}
\end{equation}

\begin{Theorem}
  \label{th:3}
  Suppose the conditions~\ref{item:4} and~\ref{item:9} hold. Let $\phi
  = \phi(x)$ be as in~\eqref{eq:21}. Then there exists a unique
  continuous function $u=u(t,x)$ on $[0,1]\times \mathbb{R}^d$ and a
  constant $N\geq 0$ such that for every $p\geq 1$
  \begin{enumerate}
  \item $t\mapsto e^{-N \phi}u(t,\cdot)$ is a H\"older continuous map
    of $[0,1]$ to $\mathbf{L}_p$ whose restriction on $(0,1]$ is
    continuously differentiable,
  \item $t\mapsto e^{-N \phi}u(t,\cdot)$ is a continuous map of
    $[0,1]$ to $\mathbf{W}^{1}_p$,
  \item $t\mapsto e^{-N \phi}u(t,\cdot)$ is a continuous map of
    $(0,1]$ to $\mathbf{W}^{2}_p$ whose restriction on $(0,1)$ is
    analytic,
  \end{enumerate}
  and such that $u=u(t,x)$ solves the Cauchy problem \eqref{eq:6} and
  \eqref{eq:7}.
\end{Theorem}
 
\begin{proof} From~\ref{item:9} we deduce the existence of $M>0$ such
  that
  \begin{displaymath}
    \abs{\frac{\partial g}{\partial x^i}}(x) \leq M e^{M\abs{x}}, \quad
    x\in \mathbb{R}^d,
  \end{displaymath}
  and, therefore, such that
  \begin{displaymath}
    \abs{g(x) - g(0)} \leq M \abs{x} e^{M\abs{x}}, \quad x \in
    \mathbb{R}^d. 
  \end{displaymath}
  Hence, for $N > M$ and a function $\phi = \phi(x)$ as
  in~\eqref{eq:21}
  \begin{displaymath}
    e^{-N\phi}g \in \mathbf{W}^{1}_{p}, \quad p \geq 1. 
  \end{displaymath}
  Hereafter, we choose the constant $N\geq 0$ so that in addition
  to~\ref{item:9} it also has the property above.

  Define the functions $\widetilde{b}^i = \widetilde{b}^i(t,x)$ and
  $\widetilde c = \widetilde c(t,x)$ so that for $t\in [0,1]$ and
  $v\in \mathbf{C}^\infty$
  \begin{displaymath}
    \widetilde A(t) \bigl(e^{-N \phi} v\bigr) = e^{-N
      \phi} A(t) v, 
  \end{displaymath}
  where
  \begin{displaymath}
    \widetilde A(t) \set  \sum_{i,j=1}^d a^{ij}(t,x) \frac{\partial^2}{\partial
      x^i \partial x^j} + \sum_{i=1}^d \widetilde{b}^i(t,x) \frac{\partial}{\partial  
      x^i} + \widetilde{c}(t,x).
  \end{displaymath}
  Direct computations show that $\widetilde{b}^i$ and $\widetilde c$
  satisfy same conditions as $b^i$ and $c$ in~\ref{item:4}. From
  Theorem~\ref{th:2} we deduce the existence of a measurable function
  $\widetilde u = \widetilde u(t,x)$ which for $p>1$ complies with the
  items \ref{item:6}--\ref{item:8} of Theorem~\ref{th:2} and solves
  the Cauchy problem:
  \begin{displaymath}
    \frac{\partial \widetilde u}{\partial t} = \widetilde
    A(t)\widetilde u + e^{-N \phi}f, \quad \widetilde
    u(0,\cdot) = e^{-N \phi}g.
  \end{displaymath}
  For $p>d$, by the classical Sobolev's embedding, the continuity of
  the map $t\mapsto \widetilde u(t,\cdot)$ in $\mathbf{W}^1_p$ implies
  its continuity in $\mathbf{C}$. In particular, we obtain that the
  function $\widetilde u = \widetilde u(t,x)$ is continuous on
  $[0,1]\times \mathbb{R}^d$.

  To conclude the proof it only remains to observe that $u=u(t,x)$
  complies with the assertions of the theorem for $p>1$ if and only if
  $\widetilde u \set e^{-N \phi}u$ has the properties just
  established. The case $p=1$ follows trivially from the case $p>1$ by
  taking $N$ slightly larger.
\end{proof}

\section{Proof of Theorem~\ref{th:1}}
\label{sec:proof-theorem}

We assume the conditions and notations of Theorem~\ref{th:1}. Observe
that without any loss in generality we can take
\begin{equation}
  \label{eq:22}
  \gamma^i(t,x) =0 \quad\text{and, hence,}\quad \widetilde{\mathbb{P}} =
  \mathbb{P}. 
\end{equation}
Indeed, by Girsanov's theorem,
\begin{displaymath}
  \widetilde W_t \set W_t - \int_0^t \gamma(s,X_s) ds
\end{displaymath}
is a Brownian motion under $\widetilde{\mathbb{P}}$. After this
substitution the equation~\eqref{eq:3} becomes
\begin{displaymath}
  dX_t = \bigl(b(t,X_t) + \sigma(t,X_t)\gamma(t,X_t)\bigr)dt +
  \sigma(t,X_t) d\widetilde W_t, \quad X_0 = x, 
\end{displaymath}
and the argument follows from the fact, that, like $b$, each component
of $\widetilde{b}\set b + \sigma \gamma$ defines a map of $[0,1]$ to
$\mathbf{L}_{\infty}$ which is analytic on $(0,1)$ and H\"older
continuous on $[0,1]$.

Hereafter we assume~\eqref{eq:22}. We fix a function $\phi$
satisfying~\eqref{eq:21}.  We also denote by $L(t)$ the infinitesimal
generator of $X$ at $t\in [0,1]$:
\begin{displaymath}
  L(t) = \frac12 \sum_{i,j=1}^d a^{ij}(t,x) \frac{\partial^2}{\partial
    x^i \partial x^j} + \sum_{i=1}^d b^i(t,x) \frac{\partial}{\partial
    x^i}, 
\end{displaymath}
where $a\set \sigma\sigma^*$ is the covariation matrix of $X$.

\begin{Lemma}
  \label{lem:4}
  There exist unique continuous functions $u=u(t,x)$ and $v^j =
  v^j(t,x)$, ${j=1,\dots,J}$, on $[0,1]\times \mathbb{R}^d$ and a
  constant $N\geq 0$ such that
  \begin{enumerate}
  \item \label{item:10} For $p \geq 1$ the maps $t\mapsto e^{-N
      \phi}u(t,\cdot)$ and $t\mapsto e^{-N \phi} v^j(t,\cdot)$ are
    \begin{enumerate}
    \item H\"older continuous maps of $[0,1]$ to $\mathbf{L}_p$ whose
      restrictions on $[0,1)$ are continuously differentiable,
    \item continuous maps of $[0,1]$ to $\mathbf{W}^1_p$,
    \item continuous maps of $[0,1)$ to $\mathbf{W}^2_p$ whose
      restrictions on $(0,1)$ are analytic.
    \end{enumerate}
  \item \label{item:11} The function $u=u(t,x)$ solves the Cauchy
    problem:
    \begin{align}
      \label{eq:23}
      \frac{\partial u}{\partial t} + ( L(t)+\beta)u &= 0, \quad t\in
      [0,1), \\
      \label{eq:24}
      u(1,\cdot) &= G,
    \end{align}
  \item \label{item:12} The function $v^j = v^j(t,x)$ solves the
    Cauchy problem:
    \begin{align}
      \label{eq:25}
      \frac{\partial v^j}{\partial t} + ( L(t)+ \alpha^j+\beta)v^j +
      uf^j + g^j&= 0, \quad t\in [0,1), \\
      \label{eq:26}
      v^j(1,\cdot) &= F^j G.
    \end{align}
  \end{enumerate}
\end{Lemma}
\begin{proof} Observe first that~\ref{item:1} on $\sigma=\sigma(t,x)$
  implies \ref{item:4} on the covariation matrix $a=a(t,x)$. The
  assertions for $u=u(t,x)$ and, then, for $v^j=v^j(t,x)$,
  $j=1,\dots,J$, follow now directly from Theorem~\ref{th:3}, where we
  need to make the time change $t\to 1-t$.
\end{proof}

Hereafter, we denote by $u=u(t,x)$ and $v^j=v^j(t,x)$,
${j=1,\dots,J}$, the functions defined in Lemma~\ref{lem:4}.

\begin{Lemma}
  \label{lem:5}
  The matrix-function $w = w(t,x)$, with $d$ rows and $J$ columns,
  given by
  \begin{equation}
    \label{eq:27}
    w^{ij}(t,x) \set  \left(u\frac{\partial v^j}{\partial x^i} -
      v^j\frac{\partial u}{\partial x^i}\right)(t,x), 
    \quad i=1,\dots,d, \;
    j=1,\dots,J,  
  \end{equation}
  has rank $d$ almost surely with respect to the Lebesgue measure on
  $[0,1]\times \mathbb{R}^d$.
\end{Lemma}
\begin{proof}
  Denote
  \begin{displaymath}
    g(t,x) \set \det (ww^*)(t,x), \quad (t,x)\in [0,1]\times \mathbb{R}^d,
  \end{displaymath}
  the determinant of the product of $w$ on its transpose, and observe
  that the result holds if and only if the set
  \begin{displaymath}
    A \set \descr{(t,x)\in [0,1]\times \mathbb{R}^d}{g(t,x) = 0} 
  \end{displaymath}
  has the Lebesgue measure zero on $[0,1]\times \mathbb{R}^d$ or,
  equivalently, the set
  \begin{displaymath}
    B \set \descr{x\in \mathbb{R}^d}{\int_0^1 \setind{A}(t,x) dt > 0} 
  \end{displaymath}
  has the Lebesgue measure zero on $\mathbb{R}^d$.

  From Lemma~\ref{lem:4} we deduce that the existence of a constant
  $N\geq 0$ such that for $p\geq 1$ the map $t\mapsto e^{-N
    \phi}g(t,\cdot)$ from $(0,1)$ to $\mathbf{W}^1_p$ is analytic and
  the same map of $[0,1]$ to $\mathbf{L}_p$ is continuous. Taking
  $p>d$, we deduce from the embedding of $\mathbf{W}^1_p$ into
  $\mathbf{C}$ that this map is also analytic from $(0,1)$ to
  $\mathbf{C}$.  It follows that if $x\in B$ then $g(t,x) = 0$ for all
  $t\in (0,1)$ and, in particular,
  \begin{displaymath}
    \lim_{t\uparrow 1} g(t,x) = 0, \quad x\in B. 
  \end{displaymath}
  Since
  \begin{displaymath}
    \norm{g(t,\cdot)- g(1,\cdot)}_{\mathbf{L}^p} = \norm{g(t,\cdot)-
      \det(ww^*)(1,\cdot)}_{\mathbf{L}^p} \to 0, \quad t\uparrow 1,
  \end{displaymath}
  the Lebesgue measure of $B$ is zero if the matrix-function
  $w(1,\cdot)$ has rank $d$ almost surely. This follows from the
  expression for $w(1,\cdot)$:
  \begin{displaymath}
    w^{ij}(1,\cdot) = G\frac{\partial (F^jG)}{\partial x^i} -
    F^jG\frac{\partial G}{\partial x^i} = G^2\frac{\partial
      F^j}{\partial x^i},
  \end{displaymath}
  and the assumption~\ref{item:2} on $F = (F^j)$ and $G$.
\end{proof}

Recall that in addition to the conditions of the theorem we also
assume~\eqref{eq:22}.

\begin{Lemma}
  \label{lem:6}
  The martingales
  \begin{align*}
    Y_t & \set \mathbb{E}[\xi|\mathcal{F}_t], \\
    R^j_t &\set \mathbb{E}[\xi \psi^j|\mathcal{F}_t], \quad
    j=1,\dots,J.
  \end{align*}
  are well-defined and have the representations
  \begin{align}
    \label{eq:28}
    Y_t & = u(t,X_t) e^{\int_0^t \beta(s,X_s)ds}, \\
    \label{eq:29}
    R^j_t & = v^j(t,X_t)e^{\int_0^t (\alpha^j+\beta)(s,X_s)ds} + Y_t
    A^j_t,
  \end{align}
  where
  \begin{displaymath}
    A^j_t \set \int_0^t  \left(f^j(s,X_s)  +
      \frac{g^j(s,X_s)}{Y_s} e^{\int_0^s 
        \beta(r,X_r)dr}\right) e^{\int_0^s \alpha^j(r,X_r)dr} ds.
  \end{displaymath}
  Moreover, for $t\in (0,1)$,
  \begin{align}
    \label{eq:30}
    dY_t &= \sum_{i,k=1}^d e^{\int_0^t \beta(s,X_s)ds}
    \left(\frac{\partial u}{\partial x^i}\sigma^{ik}\right)(t,X_t)
    d{W}^k_t, \\
    \label{eq:31}
    dR^j_t &= \sum_{i,k=1}^d e^{\int_0^t (\alpha^j+\beta)(s,X_s)ds}
    \left(\frac{\partial v^j}{\partial x^i}\sigma^{ik}\right)(t,X_t)
    d{W}^k_t + A^j_t dY_t.
  \end{align}
\end{Lemma}
\begin{proof} Assume that $Y$ and $R^j$ are actually \emph{defined}
  by~\eqref{eq:28} and~\eqref{eq:29}. From the continuity of $u$ and
  $v^j$ on $[0,1]\times \mathbb{R}^d$ we obtain that such $Y$ and
  $R^j$ are continuous processes on $[0,1]$ and from the
  expressions~\eqref{eq:24} and~\eqref{eq:26} for $u(1,\cdot)$ and
  $v^j(1, \cdot)$ that $Y_1 = \xi$ and $R^j_1 = \xi\psi^j$. Hence, to
  complete the proof we only have to show that $Y$ and $R^j$ given
  by~\eqref{eq:28} and~\eqref{eq:29} are martingales.

  Let $N\geq 0$ be the constant in Lemma~\ref{lem:4}.  Choosing $p =
  d+1$ in Lemma~\ref{lem:4} we deduce that the maps $t\mapsto e^{-N
    \phi}u(t,\cdot)$ and $t\mapsto e^{-N \phi}v^j(t,\cdot)$ of $(0,1)$
  to $\mathbf{W}^{2}_{d+1}$ are analytic and, in particular,
  continuously differentiable. This enables us to use a variant of the
  Ito formula due to Krylov, see \cite[Section 2.10,
  Theorem~1]{Krylov:80}. Direct computations, where we account
  for~\eqref{eq:23} and~\eqref{eq:25}, then yield \eqref{eq:30}
  and~\eqref{eq:31}.

  In particular, we have shown that $Y$ and $R^j$ are continuous local
  martingales. It only remains to verify their uniform integrability.
  By Sobolev's embeddings, since $t\mapsto e^{-N \phi}u(t,\cdot)$ and
  $t\mapsto e^{-N \phi}v^j(t,\cdot)$ are continuous maps of $[0,1]$ to
  $\mathbf{W}^1_{d+1}$, they are also continuous maps of $[0,1]$ to
  $\mathbf{C}$. In particular,
  \begin{displaymath}
    \abs{u(t,x)} +  \abs{v^j(t,x)} \leq e^{N(1+\abs{x})}. 
  \end{displaymath}
  Accounting for the growth properties of $f^j$, $\alpha^j$, and
  $\beta$ and denoting
  \begin{displaymath}
    B_t^j \set  \int_0^t  \frac{g^j(s,X_s)}{Y_s} e^{\int_0^s 
      (\alpha^j(r,X_r)+\beta(r,X_r))dr}ds, 
  \end{displaymath}
  we deduce the existence of a constant $N>0$ such that
  \begin{displaymath}
    \sup_{t\in [0,1]}(\abs{Y_t} + \abs{R^j_t-Y_t B_t^j}) \leq e^{N(1 + \sup_{t\in
        [0,1]} \abs{X_t})}.
  \end{displaymath}   
  As $\sup_{t\in [0,1]} \abs{X_t}$ has all exponential moments we
  obtain the martingale property for $Y$. The martingale property for
  $R^j$ follows as soon as we verify the uniform integrability of
  $(Y_t B^j_t)_{t\in [0,1]}$.

  From the growth properties of $g^j$, $\alpha^j$, and $\beta$ we
  deduce the existence of a constant $N>0$ such that
  \begin{align*}
    Y_t \abs{B^j_t} & \leq Y_t\int_0^t \frac1{Y_s} e^{N(1+\abs{X_s})}
    ds = \mathbb{E}[Y_1|\mathcal{F}_t]\int_0^t \frac1{Y_s}
    e^{N(1+\abs{X_s})}
    ds \\
    & \leq \mathbb{E}[Y_1\int_0^1 \frac1{Y_s}
    e^{N(1+\abs{X_s})}|\mathcal{F}_t], \quad t\in [0,1],
  \end{align*}
  which implies the the uniform integrability of $(Y_t B^j_t)_{t\in
    [0,1]}$ because
  \begin{displaymath}
    \mathbb{E}[Y_1\int_0^1 \frac1{Y_s}
    e^{N(1+\abs{X_s})}] = \mathbb{E}[\int_0^1 e^{N(1+\abs{X_s})}] <
    \infty. 
  \end{displaymath}
\end{proof}

We are ready to complete the proof of the theorem.  Lemma~\ref{lem:6}
implies, in particular, that
\begin{displaymath}
  \mathbb{E}[\abs{\xi} + \sum_{j=1}^J \abs{\xi \psi^j} ] < \infty, 
\end{displaymath}
and, hence, the probability measure $\mathbb{Q}$ and the
$\mathbb{Q}$-martingale $S=(S^j)$ are well-defined. Since $\xi>0$, the
measure $\mathbb{Q}$ is equivalent to $\mathbb{P}$ and the martingale
$Y$ is strictly positive. Observe that
\begin{displaymath}
  S_t \set \mathbb{E}^{\mathbb{Q}}[\psi|\mathcal{F}_t] = \frac{
    \mathbb{E}[\xi\psi|\mathcal{F}_t]}{\mathbb{E}[\xi|\mathcal{F}_t]} 
  = \frac{R_t}{Y_t}, \quad t\in [0,1].
\end{displaymath}
From~\eqref{eq:30} and~\eqref{eq:31} we deduce, after some
computations, that
\begin{equation}
  \label{eq:32}
  dS^j_t = d\frac{R^j_t}{Y_t} = e^{\int_0^t \alpha^j(s,X_s)ds}
  \frac1{u^2(t,X_t)}\sum_{i,k=1}^d(w^{ij}\sigma^{ik})(t,X_t) 
  dW^{\mathbb{Q},k}_t, 
\end{equation}
where the matrix-function $w=w(t,x)$ is defined in~\eqref{eq:27} and
\begin{displaymath}
  W^{\mathbb{Q},k}_t  \set W^k_t
  - \sum_{l=1}^d \int_0^t \left(\frac1u \frac{\partial u}{\partial
      x^l}\sigma^{lk}\right)(t,X_t) dt, 
  \quad k=1,\dots,d, \; t\in [0,1].  
\end{displaymath}
By Girsanov's theorem, $W^{\mathbb{Q}}$ is a Brownian motion under
$\mathbb{Q}$. Note that the division on $u(t,X_t)$ is safe as the
process $u(t,X_t) = Y_t e^{-\int_0^t \beta(s,X_s)ds}$, $t\in [0,1]$,
is strictly positive.

As we have already observed in Remark~\ref{rem:2}, every
$\mathbb{P}$-local martingale is a stochastic integral with respect to
$W$. This readily implies that every $\mathbb{Q}$-local martingale $M$
is a stochastic integral with respect to $W^{\mathbb{Q}}$. Indeed,
since $L\set YM$ is a local martingale under $\mathbb{P}$, there is a
predictable process $\zeta$ with values in $\mathbb{R}^d$ such that
\begin{displaymath}
  L_t = L_0 + \int_0^t \zeta_u dW_u \set L_0 + \sum_{i=1}^d \int_0^t
  \zeta^i_u dW^i_u 
\end{displaymath}
and then
\begin{displaymath}
  dM_t = d\frac{L_t}{Y_t} = \frac1{Y_t}\sum_{i=1}^d \left(\zeta^i_t - L_t
    \sum_{k=1}^d \left(\frac1u \frac{\partial u}{\partial x^k} \sigma^{ki}
    \right)(t,X_t)\right) dW^{\mathbb{Q},i}_t.   
\end{displaymath}

In view of~\eqref{eq:32}, to conclude the proof we only have to show
that the matrix-process $ ((w^*\sigma)(t,X_t))_{t\in [0,1]}$ has rank
$d$ on $\Omega\times [0,1]$ almost surely under the product measure
$dt\times d\mathbb{P}$.  Observe first that by~\eqref{eq:2} and
Lemma~\ref{lem:5} the matrix-function $w^*\sigma = (w^*\sigma)(t,x)$
has rank $d$ almost surely under the Lebesgue measure on $[0,1]\times
\mathbb{R}^d$. The result now follows from the well-known fact that
under~\ref{item:1} the distribution of $X_t$ has a density under the
Lebesgue measure on $\mathbb{R}^d$, see
\cite[Theorem~9.1.9]{StrVarad:06}.

\section*{Acknowledgments}
\label{sec:acknowledgements}

We thank Frank Riedel for introducing us to the topic of endogenous
completeness. It is a pleasure to thank our colleagues William Hrusa,
Giovanni Leoni, Dejan Slep\v{c}ev, and Luc Tartar for discussions and
Steven Shreve for the list of corrections to the previous version of
this paper.

\bibliographystyle{plainnat}

\bibliography{../bib/finance}

\end{document}